\documentclass[12pt,reqno]{amsart}
\usepackage{amsmath,a4wide}
\usepackage{xfrac}
\usepackage{stmaryrd,mathrsfs,bm,amsthm,mathtools,yfonts,amssymb,color}
\usepackage{xcolor}
\usepackage{courier}

\begingroup
\newtheorem{theorem}{Theorem}[section]
\newtheorem{lemma}[theorem]{Lemma}
\newtheorem{proposition}[theorem]{Proposition}
\newtheorem{corollary}[theorem]{Corollary}
\endgroup

\theoremstyle{definition}
\newtheorem{definition}[theorem]{Definition}
\newtheorem{remark}[theorem]{Remark}

\numberwithin{equation}{section}
\numberwithin{subsection}{section}

\newcommand\supp{{\rm spt}}
\newcommand\res{\mathop{\hbox{\vrule height 7pt width .3pt depth 0pt
\vrule height .3pt width 5pt depth 0pt}}\nolimits}
\newcommand\ser{\mathop{\hbox{\vrule height .3pt width 5pt depth 0pt
\vrule height 7pt width .3pt depth 0pt}}\nolimits}

\newcommand{\mass}{{\mathbf{M}}}
\newcommand{\cone}{{\times\hspace{-0.6em}\times\,}}
\newcommand{\bOmega}{{\mathbf{\Omega}}}
\newcommand{\de}{\partial}
\newcommand{\bI}{{\mathbf{I}}}
\def\a#1{\left\llbracket{#1}\right\rrbracket}
\newcommand{\bE}{{\mathbf{E}}}

\newcommand\weaks{{\stackrel{*}{\rightharpoonup}}\,}

\newcommand{\bB}{{\mathbf{B}}}
\newcommand{\bC}{{\mathbf{C}}}
\newcommand\proj{\mathbf{p}}
\newcommand{\bP}{\mathbf{P}}

\newcommand\N{{\mathbb N}}

\newcommand\R{{\mathbb R}}

\newcommand{\eps}{{\varepsilon}}
\newcommand{\bA}{\mathbf{A}}

\newcommand{\Lip}{{\rm {Lip}}}

\newcommand{\dist}{{\rm {dist}}}

\title[Uniqueness of tangent cones]
{Uniqueness of tangent cones for $2$-dimensional almost minimizing currents}
\author{Camillo De Lellis}
\address{Mathematik Institut der Universit\"at Z\"urich}
\email{camillo.delellis@math.uzh.ch}
\author{Emanuele Spadaro}
\address{Max-Planck-Institut f\"ur Mathematik in den Naturwissenschaften, Leipzig}
\email{spadaro@mis.mpg.de}
\author{Luca Spolaor}
\address{Mathematik Institut der Universit\"at Z\"urich}
\email{luca.spolaor@math.uzh.ch}

\begin{document}

\begin{abstract}
We consider $2$-dimensional integer rectifiable currents which are almost area minimizing and show that their tangent cones are everywhere unique. Our argument unifies a few uniqueness theorems of the same flavor, which are all obtained by a suitable modification of White's original theorem for area minimizing currents in the euclidean space. This note is also the first step in a regularity program for semicalibrated $2$-dimensional currents and spherical cross sections of $3$-dimensional area minimizing cones.
\end{abstract}

\maketitle

In this paper we consider $2$-dimensional integer rectifiable currents $T$ in the euclidean space $\R^{n+2}$ which
are almost (area) minimizing, in the following sense (for the notation and terminology we refer the reader to the textbooks 
\cite{Fed} and \cite{Sim}).

\begin{definition}\label{d:AM} An $m$-dimensional integer rectifiable current $T$ in $\R^{m+n}$ is
\emph{almost (area) minimizing} if for every $x\not\in \supp (\partial T)$ there are constants 
$C_{0}, r_0, \alpha_0 > 0$ such that 
\begin{equation}\label{e:almost minimizer2}
\|T\| (\bB_r(x)) \leq \|T + \partial S\|(\bB_r(x)) + C_{0}\, r^{m +\alpha_0}
\end{equation}
for all $0<r<r_0$ and for all integral $(m+1)$-dimensional currents $S$ supported in $\bB_r (x)$.
\end{definition}

Our aim is to extend Brian White's classical result (cf.~\cite{Wh}) on the uniqueness of tangent cones for area minimizing
$2$-dimensional currents to almost minimizers, an abstract result which can then be applied to several interesting geometric problems, recovering quickly known statements but also gaining some new ones. To state the
main theorem we introduce the current $(\iota_{x,r})_\sharp T$, where the map $\iota_{x,r}$ is given by $\R^{m+n} \ni y \mapsto \frac{y-x}{r}\in \R^{m+n}$. Recall that an area minimizing cone $S$ is an integral area minimizing current 
such that $(\iota_{0,r})_\sharp S = S$ for every $r>0$ (cf. \cite[Theorem 19.3]{Sim}).

\begin{theorem}\label{t:Uniqueness}
Assume $T$ is a $2$-dimensional integer rectifiable almost minimizing current in $\R^{n+2}$. Then for every $x\in \supp (T)\setminus \supp (\partial T)$ there is a $2$-dimensional area-minimizing cone $T_x$ with $\partial T_x = 0$ such that $T_{x,r}\to T_x$ (in the sense of currents) as $r\downarrow 0$.
\end{theorem}

From this theorem we conclude three interesting corollaries as special cases.

\begin{definition}\label{d:semicalibrated}
Let $\Sigma \subset R^{m+n}$ be a $C^2$ submanifold and $U\subset \R^{m+n}$ an open set.
\begin{itemize}
  \item[(a)] An $m$-dimensional integral current $T$ with finite mass and $\supp (T)\subset \Sigma\cap U$ is area-minimizing in $\Sigma\cap U$
if $\mass(T + \partial S)\geq \mass(T)$ for any $m+1$-dimensional integral current $S$ with $\supp (S) \subset \subset \Sigma\cap U$.
 \item[(b)] A semicalibration (in $\Sigma$) is a $C^1$ $m$-form $\omega$ on $\Sigma$ such that 
  $\|\omega_x\|_c \leq 1$ at every $x\in \Sigma$, where $\|\cdot \|_c$ denotes the comass norm on $\Lambda^m T_x \Sigma$. 
  An $m$-dimensional integral current $T$ with $\supp (T)\subset \Sigma$ is {\em semicalibrated} by $\omega$ if $\omega_x (\vec{T}) = 1$ for $\|T\|$-a.e. $x$.
  \item[(c)]  An $m$-dimensional integral current $T$ supported in $\partial \bB_R (x) \subset \R^{m+n}$ is a {\em spherical cross-section of an area-minimizing cone} if ${x\cone T}$ is area-minimizing. 
\end{itemize}
\end{definition}

In all these cases, when $m=2$, we conclude the uniqueness of tangent cones  from Theorem \ref{t:Uniqueness} and the following 

\begin{proposition}\label{p:AM}
Under the assumptions of Definition \ref{d:semicalibrated}, any $m$-dimensional current $T$ as in (a), (b) or (c) is almost minimizing in the sense of Definition \ref{d:AM} and therefore, by Theorem \ref{t:Uniqueness}, it has a unique tangent cone at every $x\not\in \supp (\partial T)$ if $m=2$.
\end{proposition}

In this paper we will consider only Riemannian submanifolds $\Sigma$ of some euclidean space. However,
since all the statements are local, by Nash's isometric embedding theorem we can infer the same conclusions in any abstract Riemannian manifold which is sufficiently regular: in particular, since we need $C^2$ regularity in the embedded case, in the abstract setting we can derive the same consequences when the Riemannian metric is $C^{2,\alpha}$ for some positive $\alpha$.

The uniqueness of tangent cones for $2$-dimensional area-minimizing currents in Riemannian manifolds (case (a)) has been proved first by Chang in \cite{Chang}. The same statement for semicalibrated integral $2$-dimensional cycles (case (b)) has been shown more recently by Pumberger and Rivi\`ere in \cite{PuRi}. As far as we know the result for
spherical cross sections of $3$-dimensional area-minimizing cones is instead new.
Our motivation comes in fact from the interior regularity theory for all these objects: in a series of forthcoming papers we will prove that any $2$-dimensional current as in (a), (b) or (c) is either a regular submanifold in the interior or has isolated singularities. The latter result is due to Chang in case (a), but as far as we know the details of one crucial step in Chang's proof have never appeared. It is instead due to Bellettini and Rivi\`ere for a particular case of (b), see \cite{BeRi}: in their theorem $\Sigma$ is the $5$-dimensional standard sphere and the semicalibrated currents are the so-called special legendrian cycles. In all the other situations such regularity theorem would be a new
result and this note is the first step of our program to prove it.

In codimension $1$ the uniqueness of tangent cones is known at isolated singularities thanks to the pioneering work of Simon, cf. \cite{Simon}.
The uniqueness of tangent cones is widely open in dimension higher than $2$ and general codimension. Some interesting higher dimensional cases have been recently covered by Bellettini in \cite{Be1, Be2}.

\subsection{Acknowledgments} The research of Camillo De Lellis and Luca Spolaor has been supported by the ERC grant
agreement RAM (Regularity for Area Minimizing currents), ERC 306247. 
The authors are warmly thankful to Bill Allard and Guido de Philippis for several important discussions.


\section{Proof of Proposition \ref{p:AM}}

We start remarking that in the semicalibrated case we do not loose any generality if we consider $\Sigma$ to be the ambient euclidean space.
We then point out one elementary variational property of semicalibrated currents and spherical cross sections of minimizing cones. We will use the notation $\bI_m (\R^{m+n})$ for the space of integral currents in $\R^{m+n}$ (cf. \cite[Section 4.1.24]{Fed}).

\begin{lemma}\label{l:riduzione}
Let $k\in \mathbb N\setminus\{ 0\}$, $\eps_0 \in [0,1]$, $\Sigma\subset \R^{m+n}$ be a $C^{k+1, \eps_0}$ $m+\bar{n}$-dimensional submanifold, $V\subset \R^{m+n}$ an open subset and 
$\omega$ a $C^{k, \eps_0}$ $m$-form on
$V\cap \Sigma$. If $T$ is a cycle in $V\cap \Sigma$ semicalibrated by $\omega$, then $T$ is semicalibrated in $V$ by a $C^{k, \eps_0}$ form $\tilde{\omega}$. 
\end{lemma} 

\begin{proof}
The argument is straightforward: we just need to extend $\omega$ to a form $\tilde{\omega}$ on the open set $V$ in such a way that 
$\|\tilde{\omega}_x\|_c \leq 1$ for every $x$ and the regularity of $\omega$ is preserved. Without loss of generality it suffices to do this on a tubular neighborhood $U$ of $\Sigma \cap V$ on which there is a $C^{k, \eps_0}$ orthogonal projection $\proj: U\to \Sigma \cap U$ (we then multiply this extension by a function $\varphi\in C^\infty_c (U)$ which is identically $1$ on $\Sigma$ and satisfies $0\leq \varphi\leq 1$; the resulting form can then be extended to $V$ by setting it equal to $0$ where it is not yet defined). For
$x\in U$ we set $y := \proj (x) \in \Sigma$ and let $\proj_y : \R^{m+n} \to T_y \Sigma$ be the orthogonal projection. We then set 
$\tilde{\omega}_x (v_1, \ldots, v_m) = \omega_y (\proj_y (v_1), \ldots, \proj_y (v_m))$. Observe that $\tilde{\omega}$ {\em is not} $\proj^\sharp \omega$
(in general the latter would not satisfy $\|\tilde{\omega}_x\|_c \leq 1$).  
\end{proof}

\begin{proposition}\label{p:quasi_minimalita}
Let $T$ be as in Definition \ref{d:semicalibrated} (b) (in which case we assume $\Sigma = \R^{m+n}$) or (c). Then there is a constant $\bOmega$ such that
\begin{equation}\label{e:Omega}
\mass (T) \leq \mass (T + \partial S) + \bOmega\, \mass (S)\qquad \forall S\in {\bf I}_{m+1} (\R^{m+n})\quad \mbox{with compact support.}
\end{equation}
Moreover, $\bOmega \leq \|d\omega\|_0$ in case (b) and $\bOmega \leq (m +1)R^{-1}$ in case (c).

Moreover, if $\chi\in C^\infty_c (\R^{m+n}\setminus \supp (\partial T), \R^{m+n})$, we have
\begin{align}
&\delta T (\chi) = T (d\omega \ser \chi) &\mbox{in case (b),}\label{e:var_prima_(b)}\\
&\delta T (\chi) = \int m R^{-1}\, x \cdot \chi (x)\, d\|T\| (x) &\mbox{in case (c).} \label{e:var_prima_(c)}
\end{align}
\end{proposition}

\begin{proof} We first prove \eqref{e:Omega}.
Assume we are in case (c). Without loss of generality we can assume $x=0$ and $R=1$. Therefore fix $S$ compactly supported and consider $W = T + \partial S$. Next, let $p: \mathbb R^{m+n} \to  \overline{\bB}_1 (0)$ be the orthogonal projection and set $S' = p_\sharp S$ and $W' := p_\sharp W = T + \partial p_\sharp S$ (where the latter identity holds because $\supp (T)\subset \partial \bB_1 (0)$). 
The current $Z := 0\cone W' - S'$ is then a competitor for the minimality of $0\cone T$ and observe, moreover, that since $\supp (W')\subset \overline{B}_1 (0)$, we have $\mass (Z) \leq (m+1)^{-1} \mass (W')$. Then we have
\begin{align*}
0 \leq & (m+1) (\mass (Z) -\mass (0\cone T)) \leq \mass (W') - \mass (T)  + (m+1) \mass (S')\\
\leq &\mass (W) - \mass (T) + (m+1) \mass (S)\, .  
\end{align*}
In case (b), if $\omega$ is the semicalibrating form, we can then estimate
\[
\mass (T) = T (\omega) = W (\omega)-  \partial S (\omega) \leq \mass (W) - S (d\omega)\leq \mass (W) + \|d\omega\|_0 \mass (S)\, .
\]

Next, \eqref{e:var_prima_(c)} is simply the stationarity of $T$ in $\partial \bB_1 (0)$. As for \eqref{e:var_prima_(b)} the formula seems new
in the literature and we provide here a simple proof. Fix $\chi$ and consider the maps $\Phi_t (x) := x + t \chi (x)$ and $\Lambda (t,x) = \Phi_t (x)$. We then denote by $\llbracket0, \eps\rrbracket$ the current in $\mathbf{I}_1 (\R)$ induced by the oriented segment 
$\{t: 0\leq t\leq \eps\}$.
We define $T_\eps := (\Phi_\eps)_\sharp T$ and $S_\eps:= \Lambda_\sharp (\llbracket0, \eps\rrbracket\times T)$. We then have $\partial S_\eps = T_\eps - T$ and hence
\begin{align}
\mass (T_\eps)-\mass (T) &\geq T_\eps (\omega) - T (\omega) = S_\eps (d\omega) = \llbracket0, \eps\rrbracket\times T (\Lambda^\sharp d\omega) =: h (\eps)\, .
\end{align}
Since $h$ is $C^1$ and $h(0)=0$, by a Taylor expansion we conclude $\eps \delta T (\chi) \geq \eps h'(0) + o (\eps)$. On the other hand, since the latter inequality is valid for both positive and negative $\eps$, we infer $\delta T (\chi) = h'(0)$. We thus only need to show the identity $h'(0) = T (d\omega \ser \chi)$. Consider the set of ordered multiindices $I = \{1\leq i_1 < i_2< \ldots < i_{m+1}\}$ and let $d\omega = \sum f_I dx^I$, where
$dx^I = dx^{i_1}\wedge \ldots \wedge dx^{i_{m+1}}$. We then have 
\[
(\Lambda^\sharp d \omega)_{(x,t)} = \sum f_I (\Phi_t (x))  d\Phi^{i_1}_t \wedge \ldots \wedge d\Phi^{i_{m+1}}_t\, .
\]
Next, we will denote by $o(1)$ any continuous function of $x$ and $t$ which vanish at $t=0$ and we let $\pi: \R\times \R^{m+n} \to \R^{m+n}$ be the projection $\pi (t,x)=x$.
Since $\Phi (0,x) =x$ and $f_I$ is continuous we conclude
\begin{align}
& (\Lambda^\sharp d \omega)_{(x,t)} = \sum f_I (x) d\Phi^{i_1}_t \wedge \ldots \wedge d\Phi^{i_{m+1}}_t + o (1) =\nonumber\\
& \sum_I  f_I (x) \Big(dx^I + \sum_{1\leq j \leq k+1} f_I (x) \chi^{i_j} (x) dx^{i_1} \wedge \ldots \wedge dx^{i_{j-1}} \wedge dt \wedge dx^{i_{j+1}}
\wedge \ldots \wedge dx^{m+1}\Big) + o (1)\nonumber\\
&= \pi^\sharp d\omega + dt \wedge \sum_I f_I (x) \sum_j (-1)^j \chi^{i_j} (x) dx^{i_1} \wedge \ldots \wedge dx^{i_{j-1}} \wedge dx^{i_{j+1}}
\wedge \ldots \wedge dx^{m+1} + o(1)\nonumber\, .
\end{align}
Thus,
\[
(\Lambda^\sharp d \omega)_{(x,t)} = \pi^\sharp d\omega + dt \wedge \pi^\sharp (d\omega \ser \chi) + o(1)\, . 
\]
In particular, since $d\omega$ is orthogonal to $dt$, we have
$\llbracket0, \eps\rrbracket\times T (\pi^\sharp d\omega) =0$. Thus we can write
\[
h (\eps) = \llbracket0, \eps\rrbracket\times T (dt \wedge \pi^\sharp (d\omega \ser \chi)) + o (1) \eps \mass (T) = \eps T (d\omega \ser \chi) + o (\eps)\, ,
\]
 from which we finally conclude $h'(0) = T (d\omega \ser \chi)$.
\end{proof}

\begin{proof}[Proof of Proposition \ref{p:AM}] 
{\bf Case (a).} Consider $x\in \Sigma$ and a ball $\bB_r (x)\subset \R^{m+n}$. If $\bar{r}$ is sufficiently small there is a well-defined $C^1$ orthogonal projection $\proj: \bB_{\bar{r}} (x) \to \Sigma$ with the property that $\Lip (\proj) \leq 1 + C \bA r$, where $C$ is a geometric constant and $\bA$ denotes the $L^\infty$ norm of the second fundamental form of $\Sigma$. Consider $T$ area-minimizing in $\Sigma$ and assume
$\bar{r}< \dist (x, \supp (\partial T))$. Let $r\leq \bar{r}$ and $S\in \bI_{m+1} (\R^{m+n})$ be such that $\supp (S)\subset \bB_r (x)$. We set $W:= T + \partial S$. If $\|W\| (\bB_r (x)) \geq \|T\| (\bB_r (x))$ there is nothing to prove, otherwise by the standard monotonicity formula we have $\|W\| (\bB_r (x)) \leq \|T\| (\bB_r (x)) \leq C r^m$. 
Then $W':= \proj_\sharp W$ is an admissible competitor for the minimality property of $T$ and we have 
\[
\|T\| (\bB_r (x)) \leq \|W'\| (\bB_r (x)) \leq (\Lip (\proj))^m \|W\| (\bB_r (x)) \leq \|W\| (\bB_r (x)) + C r^{m+1}\, .
\]

\medskip

{\bf Case (b)\&(c).} First observe that, by Lemma \ref{l:riduzione}, in case (b) we can assume, w.l.o.g., that $\Sigma = \R^{m+n}$. Fix $r< \dist (x, \supp (\partial T))$ and let $S\in \bI_{m+1} (\R^{m+n})$ be such that $\supp (S) \subset \bB_r (x)$. As above, either $\|W\| (\bB_r (x)) \geq \|T\| (\bB_r (x))$, in which case
there is nothing to prove, otherwise by the standard monotonicity formula we have $\|W\| (\bB_r (x)) \leq \|T\| (\bB_r (x)) \leq C r^m$ (observe that, by \eqref{e:var_prima_(b)} and \eqref{e:var_prima_(c)}, $T$ induces a varifold with bounded mean curvature, which in turn implies Allard's monotonicity formula, cf. \cite[Section 17]{Sim}). 
In the latter case, by the isoperimetric inequality there exists $S'\in \bI_{m+1} (\R^{m+n})$ such that
\[
\partial S'=\partial S\quad \text{and}\quad \mass(S')\leq C r^{m+1}\, .
\]
Applying now \eqref{e:Omega} to this current $S'$ we get the desired conclusion, with $C_1=C\bOmega$.
\end{proof}

\begin{remark}\label{r:costanti}
Observe that we have achieved \eqref{e:almost minimizer2} with any fixed $r_0< \frac{1}{2} \dist (x, \supp (\partial T))$, $\alpha_0=1$ and $C_0 = C\bA$, in case (a), $C_0 = C\bOmega$, in the
cases (b) and (c), where the constant $C$ depends only upon $\|T\| (B_{2r_0}) (x)$.
\end{remark}

\section{Two technical lemmas}
It is known that the almost minimizing condition of Definition \ref{d:AM} is alone sufficient to derive a monotonicity formula. However, we have been unable to find a reference and we therefore provide the proof below.
Note also that in the geometric cases (a), (b) and (c), a more precise form of the monotonicity formula could be derived directly appealing to the fact that the corresponding induced varifolds have bounded mean curvature. 

\begin{proposition}[Almost Monotonicity]\label{p:AMO}
Let $T\in {\bf I}_m (\R^{m+n})$ be an almost minimizer and $x\in \supp (T)\setminus \supp (\partial T)$. There are constants $C_{02}, \bar{r}, \alpha_0>0$ such that 
\begin{equation}\label{e:monotonicity formula}
\int_{\bB_r (x)\setminus \bB_s(x)} \frac{|(z-x)^\perp|^2}{|z-x|^{m+2}} d \|T\|(z)
\leq C_{02}\Bigl( \frac{\|T\|(\bB_r(x))}{\omega_k\,r^m}
-\frac{\|T\|(\bB_s(x))}{\omega_m\,s^m} +\,r^{\alpha_0} \Bigr)
\end{equation}
for all $0<s<r<\bar{r}$ (in \eqref{e:monotonicity formula} $(z-x)^\perp$ denotes the projection of the vector 
$z-x$ on the orthogonal complement of the approximate tangent to $T$ at $z$). In particular, the function $\displaystyle{r\to  \frac{\|T\|(\bB_r(x))}{\omega_k\,r^m}+r^{\alpha_0}}$ is nondecreasing.
\end{proposition}

\begin{proof}[Proof of Proposition \ref{p:AMO}] Assume without loss of generality $x=0$. For a.e. $r$ the current $\de (T\res \bB_r)$ is integral (cf. \cite[Section 28]{Sim}) and we have, by \eqref{e:almost minimizer2} with $W=0\cone \de (T\res \bB_r)$,
\begin{equation}\label{e:comp_cono}
\|T\|( \bB_r)\leq \|W\|(\bB_r)+C_{0}r^{m+\alpha_0}=\frac{r}{m}\mass(\de(T\res \bB_r))+C_{0}r^{m+\alpha_0}.
\end{equation}
Set $f(r):=\|T\|(\bB_r)$ and observe that $f$ is an nondecreasing function and so a function of bounded variation. As such it has left and right limits at each point and in fact $f (r) = f(r^-)$. In particular we can decompose its distributional derivative $Df$, which is a nonnegative measure, as $Df = f' \mathscr{L} + \mu_s$, where $\mathscr{L}$ denotes the Lebesgue one-dimensional measure and $\mu_s$ is the singular part of $Df$. 
We multiply \eqref{e:comp_cono} by $m r^{-m-1}$ and add $\frac{f'(r)}{r^m} + \frac{\mu_s}{r^m}$:
\[                
\frac{\mu_s}{r^m} + \frac{1}{r^m}f'(r)-\frac{1}{r^m}\mass(\de (T\res \bB_r))\leq \frac{Df}{r^m}-\frac{mf(r)}{r^{m+1}}+C_{01} r^{\alpha_0 -1}.
\]
Integrating on the interval $[s,r[$ (where $r_0>r>s$) we reach
\[
\underbrace{\int_{[s,r[} \frac{1}{\rho^m} d\mu_s (\rho)}_{I^s} + \underbrace{\int_s^r\frac{1}{\rho^m}(f'(\rho)-\mass(\de (T\res \bB_\rho)))\,d\rho}_{I^a}\leq \frac{f(r)}{r^m}-\frac{f(s)}{s^m}+C_0 r^{\alpha_0}.
\]
To conclude we only need to prove that $I := I^s+I^a$ bounds the left hand side of \eqref{e:monotonicity formula}. Denote by $x^\parallel$ the projection of $x$ on the approximate tangent space to $T$ at $x$. Recall first (cf. \cite[eq. (28.6)]{Sim}) that 
\[
T_\rho := \langle T, |\cdot|,\rho\rangle=\de (T\res \bB_\rho)-(\de T)\res \bB_\rho=\de (T\res \bB_\rho)\, .
\]
Next introduce the Borel set $E:= \{|x^{\parallel}| > 0\}$ and its complementary $E^c$ and recall that, by the coarea formula (cf. \cite[Lemma 28.1 \& Lemma 28.5]{Sim}), for any Borel map $g$ we have
\begin{equation}\label{e:coarea}
\int_{\bB_r\setminus \bB_s} g (y) \frac{|y^\parallel|}{|y|} d\|T\| (y) = \int_s^r \int g(x) d\|T_\rho\| (x)\, d\rho\,  .
\end{equation}
Let $R$ be the countable rectifiable set such that $\|T\| = \Theta (T,x) \mathcal{H}^m\res R$. It then follows from the slicing theory that
$\|T_\rho\| = \Theta (T,x) \mathcal{H}^{m-1} \res (R\cap \partial \bB_\rho)$ for a.e. $\rho$ and thus inserting $g = {\mathbf 1}_{E^c}$ in \eqref{e:coarea} above we derive
\begin{equation}\label{e:Sard}
\mathcal{H}^{m-1} (E^c\cap \partial \bB_\rho) \leq \|T_\rho\| (E^c)=0\qquad \mbox{for a.e. $\rho$.} 
\end{equation}
Thus, since $|x^\parallel|>0$ for every $x\in (\bB_r\setminus \bB_s) \cap E$, we conclude
\begin{align}
I^a = &\int_s^r \frac{1}{\rho^m} \int_E \frac{|x|-|x^\parallel|}{|x^\parallel|} d \|T_\rho\| (x)\, d\rho
=  \int_s^r \frac{1}{\rho^m} \int_E \frac{|x|^2-|x^\parallel|^2}{|x|^\parallel (|x|+|x^\parallel|)} d \|T_\rho\| (x)\, d\rho\nonumber \\
\geq & \int_s^r \frac{1}{2\rho^{m+2}} \int_E \frac{|x^\perp|^2|x|}{|x^\parallel|} d \|T_\rho\| (x)\, d\rho
= \int_{(\bB_r\setminus \bB_s) \cap E} \frac{|x^\perp|^2}{2|x|^{m+2}} d\|T\| (x)\, .\label{e:ass_continua}
\end{align}
Now observe that on $E^c$, the complement of $E$, we have $|x^\perp| =|x|$ and thus
\begin{equation}\label{e:singolare_1}
\int_{(\bB_r\setminus \bB_s) \cap E^c} \frac{|x^\perp|^2}{2|x|^{m+2}} d\|T\| (x) = \int_{(\bB_r\setminus \bB_s) \cap E^c} 
\frac{1}{2|x|^{m}} d\|T\| (x)\, .
\end{equation}
Next, denote by $S$ the set of radii $r$ such that $\mathcal{H}^{m-1} (E^c\cap \partial \bB_r)>0$. We then must have
\[
\|T\| (E^c \cap (\bB_\rho \setminus \bB_\tau)) \leq \|T\| \left( \cup_{s\in S \cap [\tau, \rho[} \partial \bB_s\right) \leq Df (S\cap [\tau, \rho[)
\stackrel{\eqref{e:Sard}}{\leq} \mu_s ([\tau, \rho[)
\] 
for every $0<\tau < \rho$
(in fact the inequalities above are all identities, but this is not really needed). Thus for every $N\in \mathbb N \setminus 0$ we can estimate
\begin{align*}
\int_{(\bB_r\setminus \bB_s) \cap E^c} 
\frac{1}{2|x|^{m}} d\|T\| (x) \leq \sum_{i=1}^N \frac{1}{2s_{i-1}^m} \|T\| (E^c \cap (\bB_{s_i}\setminus \bB_{s_{i-1}})) \leq \sum_{i=1}^N \frac{1}{2s_{i-1}^m} \int_{[s_{i-1}, s_i[} d\mu_s 
\end{align*}
where $s_i := s + \frac{i}{N} (r-s)$. In particular letting $N\uparrow \infty$ we concude
\begin{equation}\label{e:singolare_2}
\int_{(\bB_r\setminus \bB_s) \cap E^c} 
\frac{1}{2|x|^{m}} d\|T\| (x) \leq \int_{[s,r[} \frac{1}{2\rho^m} d\mu_s (\rho) = I^s\, .
\end{equation}
From \eqref{e:ass_continua}, \eqref{e:singolare_1} and \eqref{e:singolare_2} we conclude that $I^a+I^s$ bounds the right hand side of \eqref{e:monotonicity formula}. 
\end{proof}

The following proposition tells us that if a current as in (a), (b) or (c) in Definition \ref{d:semicalibrated} is suitably decomposed, then each element of the decomposition is again respectively of type (a), (b) or (c).

\begin{proposition}\label{p:decompose} Let $T$ be as in Definition \ref{d:semicalibrated}$(\diamondsuit)$, with $\diamondsuit = a,b$ or $c$, and suppose that there are $x\in \supp (T)\setminus \supp (\partial T)$, $\bar{r}>0$ and $J$ currents $T^1, \ldots, T^J$ such that 
\[T\res \bB_{\bar r}(x)=\sum_{j=1}^JT^j,\quad
 \de T^j\res \bB_{\bar r}(x)=0
\quad\text{and}\quad
 \|T\|(\bB_{\bar r}(x))=\sum_{j=1}^J \|T^j\|(\bB_{\bar r}(x))\,.
\]
Then each $T^j$ satisfies $(\diamondsuit)$ in Definition \ref{d:semicalibrated}.
\end{proposition}

\begin{proof} We divide the proof in the three cases of Definition \ref{d:semicalibrated}.

\medskip

{\bf (a)} Suppose by contradiction that there exist $j\in\{1,\dots,J\}$ and $S\in\bI_{m+1} (\Sigma)$ with $\supp (T) \subset \bB_{\bar r} (x)$ such that
$\mass (T^j\res \bB_{\bar{r}} (x)) > \mass (T^j\res \bB_{\bar{r}} (x) + \partial S)$. Then it is straightforward to check that $\mass (T\res \bB_{\bar{r}} (x)+\partial S) < \mass (T\res \bB_{\bar{r}} (x))$, which contradicts the minamility of $T$.

\medskip

{\bf (b)} By contradiction, suppose there exists $j\in \{1,\dots,J\}$ such that $T^j$ is not semicalibrated by $\omega$. Assume $j=1$. Then since $\|\omega\|_c\leq 1$, we have $T^1(\omega)<\|T^1\|(B_{\bar r}(x))$ and $T^j(\omega)\leq\|T^j\|(B_{\bar r}(x))$, for every $j\in \{2,\dots,J\}$. It follows that
\[
\|T\|(\bB_{\bar r}(x))=T(\omega)=\sum_{j=1}^JT^j(\omega)<\sum_{j=1}^J \|T^j\| (\bB_{\bar r} (x))=\|T\|(\bB_{\bar r} (x)))
\]
which gives a contradiction and concludes the proof.

\medskip

{\bf (c)} Without loss of generality we can assume $x=0$ and $R=1$. Again by contradiction assume there exist $j\in\{1,\dots,J\}$ and $S\in\bI_{m+1}(\R^{m+n})$ such that $\de (S\res C)=\de(0\cone T^j\res C)$ and $\mass(S\res C)<\mass (0\cone T^j\res C)$, where 
\[
C:=\big\{\lambda z\,:\, z\in B_{\bar r}(x)\cap \partial \bB_1 (0),\, \lambda\in ]0,1[\big\}.
\] 
We can assume $j=1$. Notice also that 
\begin{equation}\label{e:cone to boundary}
\mass ((0\cone T)\res C)=\frac{1}{m} \|T\|(\bB_{\bar r}(x))=\frac{1}{m}\sum_{j=1}^J\|T^j\|(\bB_{\bar r}(x))=\sum_{j=1}^J\mass ((0\cone T^j)\res C).
\end{equation}
Then we have
\begin{align*}
\mass((0\cone T)\res C)
&\leq \mass \Bigl(\Bigl(S+\sum_{j=2}^J0\cone T^j\Bigr)\res C\Bigr)\leq\mass(S\res C)+\mass\Bigl(\sum_{j=2}^J(0\cone T^j)\res C\Bigr)\\
&<  \mass ((0\cone T^1)\res C)+\mass\Bigl(\sum_{j=2}^J(0\cone T^j)\res C\Bigr)\stackrel{\eqref{e:cone to boundary}}{=}\mass((0\cone T)\res C). 
\end{align*}
The latter is a contradiction and thus completes the proof.
\end{proof}

\section{A generalization of White's epiperimetric inequality and the proof of Theorem \ref{t:Uniqueness}}

In this section we show how Theorem \ref{t:Uniqueness} follows from a suitable epiperimetric inequality due to
Brian White. However, since we prove a more accurate version of Theorem \ref{t:Uniqueness}, we state it again
more precisely in the following theorem. From now, for any given $R\in \bI_m (\R^{m+n})$ we define
$\mathcal{F} (R) := \inf \{ \mass (Z) + \mass (W) : Z\in \bI_m, W\in \bI_{m+1}, Z + \partial W = R\}$.

\begin{theorem}[Uniqueness of tangent cones for almost minimizers]\label{t:Uniqueness2} Let $T\in \bI_2(\R^{n+2})$ be an almost minimizer. Then there is a $\gamma_0>0$, $J$ $2$-dim.~distinct planes $\pi_i$, each pair of which intersect only at $0$, and $J$ integers $n_i$ such that, if we set $S:= \sum_i n_i \a{\pi_i}$, then
\begin{align}
&\mathcal{F} \big((T_{x,r}-S)\res \bB_1\big)\leq C_{11} \, r^{\gamma_0},\label{e:roc1}\\
&\dist\big(\supp(T\res \bB_r(x)), \supp (S)\big)\leq C_{11} \, r^{1+\gamma_0}.\label{e:roc2}
\end{align}
Moreover, there are $\bar{r}>0$ and $J\geq 1$ currents $T^j\in {\bf I}_2 (\bB_{\bar r} (x))$ such that 
\begin{itemize}
\item[(i)] $\partial T^j \res \bB_{\bar r} (x) =0$ and each $T^j$ is an almost minimizer;
\item[(ii)] $T\res \bB_{\bar r} (x) = \sum_j T^j$ and $\supp (T_j)\cap \supp (T_i) = \{x\}$ for every $i\neq j$;
\item[(iii)] $n_j \llbracket \pi_j \rrbracket$ is the unique tangent cone to each $T^j$ at $x$.
\end{itemize}
\end{theorem}

From the latter theorem, Proposition \ref{p:AM} and Proposition \ref{p:decompose} we conclude

\begin{corollary}\label{c:Uniqueness-abc}
Let $T$ be as in Definition \ref{d:semicalibrated}$(\diamondsuit)$, with $\diamondsuit =a,b$ or $c$, and $x\in \supp (T)\setminus \supp (\partial T)$. Then all the conclusions of Theorem \ref{t:Uniqueness2} hold for $T$ and moreover each $T^j$ satisfies Definition \ref{d:semicalibrated}$(\diamondsuit)$.
\end{corollary}

\subsection{White's epiperimetric inequality and its generalization} As already mentioned, the key ingredient in the proof of Theorem \ref{t:Uniqueness2} is a suitable generalization of White's epiperimetric inequality \cite{Wh}. We record the main ingredient of White's argument in the following lemma. Since however the paper \cite{Wh} does not state this lemma explicitely, we provide in the last
section a brief argument, referring to propositions and lemmas which are instead explicitely stated in \cite{Wh} (the only difference is 
in a technical point, namely the estimate \eqref{e:Fourier}, for which we point out a shorter argument).

\begin{lemma}\label{l:construction}
Let $S \in {\bf I}_2(\R^{n+2})$ be an area minimizing cone.
There exists a constant $\eps_{13}>0$ with the following property. If
$R:= \de (S\res \bB_1)$ and $Z\in {\bf I}_1(\de \bB_1)$ is a cycle with
\begin{itemize}
\item[(i)] $\mathcal{F}(Z-R)<\eps_{13}$, 
\item[(ii)] $\mass(Z)-\mass(R)<\eps_{13}$,
\item[(iii)] $\dist\big(\supp (Z),\supp (R)\big)<\eps_{13}$,
\end{itemize}
then there exists $H\in {\bf I}_2(\bB_1)$ such that $\partial H=Z$
and
\[
\|H\|(\bB_1)-\|S\|(\bB_1)\leq (1-\eps_{13})\big[\|0\cone Z\|(\bB_1)-\|S\|(\bB_1)\big]. 
\]
\end{lemma}

A simple compactness argument allows us to generalize this lemma in the following sense.

\begin{proposition}\label{p:generalized epi}
Let $S \in {\bf I}_2(\R^{n+2})$ be an area minimizing cone.
For every $C_{12}>0$ there exists a constant $\eps_{11}>0$, depending only on the constants $C_{01}$ and $\alpha_0$
of Definition \ref{d:AM} and upon $S$, with the following property.
Assume that $T \in {\bf I}_2(\R^{n+2})$ is an
almost minimizer with $0\in \supp(T)$ and set $T_\rho := (\iota_{0, \rho})_\sharp T$.
If $r$ is a positive number with 
\begin{itemize}
\item $0< 2\,r < \min\{2^{-1}\dist(0,\supp(\de T)), 2\eps_{11}\}$,
\item $\mathcal{F}\big((T_{2r}-S)\res \bB_1\big) < 2\,\eps_{11}$, $\|T\| (\bB_{2r}) \leq C_{12} r^2$
\item and $\de (T\res \bB_r) \in {\bf I}_1(\R^{n+2})$,
\end{itemize} 
then
\begin{equation}\label{e:generilez epi}
\|T_r\|(\bB_1) - \|S\|(\bB_1) \leq \left(1-\eps_{12}\right) \Big( \|0\cone \de(T_r\res \bB_1)\| (\bB_1) - \|S\|(\bB_1)\Big) + \bar c\, r^{\alpha_0} .
\end{equation}
$\bar{c}$ depends only on $C_{01}$, $\alpha_0$ and $\Theta (0, S)$
and $\eps_{12}>0$ is any number smaller than some $\bar{\eps}>0$, which also depends on $C_{0}$, $\alpha_0$ and $\Theta (0, S)$. Moreover $\bar{c}$ 
 depends linearly on $C_{01}$. In particular, if $T$ is as in Definition \ref{d:semicalibrated}, then $\alpha_0=1$ and:
$\bar{c}$ depends linearly on $\bA:=\|A_\Sigma\|_\infty$ in case (a), it depends linearly on $\bOmega:= \|d\omega\|_\infty$ in case (b) and it quals $C_0 R^{-1}$ for some geometric constant $C_0$ in case (c) (in the sense of Remark \ref{r:costanti}).
\end{proposition}

\begin{proof} 
We argue by contradiction and assume there exist
sequences of almost minimizers $(T^k)_{k\in \N}\subset {\bf I}_2(\R^{2+n})$
and radii  $r_k\downarrow 0$ with $0< 2\,r_k < \dist(0,\supp(\de T^k))$
such that $R^k:= (T^k)_{r_k}$ satisfies
$\mathcal{F}((R^k-S)\res \bB_2) < \frac{1}{k}$ and
\begin{equation}\label{e:hyp1}
\|R^k\| (\bB_1)- \|S\|(\bB_1) > \Bigl(1-\frac{1}{k}\Bigr)
\Big(\|0\cone \de(R^k\res \bB_1)\|(\bB_1)-\|S\|(\bB_1)\Big)+ k \, r_k^{\alpha_0}.
\end{equation}
It is important to notice that, in contradicting the statement of Proposition~\ref{p:generalized epi}, the currents
$T^k$ satisfy \eqref{e:almost minimizer2} for some constants $C_{0}$ and $\alpha_0$ which
are fixed, i.e. {\em independent of $k$}.
First of all, without loss of generality we can assume
\begin{equation}\label{e:wlog}
\|0\cone \de(R^k\res \bB_1)\|(\bB_1)-\|S\|(\bB_1) \geq 0\, ;
\end{equation}
indeed if $\|0\cone \de(R^k\res \bB_1)\|(\bB_1)-\|S\|(\bB_1)<0$ we could
use the almost minimality and the appropriate rescaling to conclude
\begin{align*}
\|R^k\| (\bB_1)- \|S\|(\bB_1) &\leq \|0\cone \de(R^k\res \bB_1)\|(\bB_1)-\|S\|(\bB_1) + C_1 r_k^{\alpha_0}\\
&\leq  \Bigl(1-\frac{1}{k}\Bigr)
\Big(\|0\cone \de(R^k\res \bB_1)\|(\bB_1)-\|S\|(\bB_1)\Big)+ C_1 \, r_k^{\alpha_0}\, ,
\end{align*}
contradicting \eqref{e:hyp1} for $k$ large enough.

Observe that we have a uniform bound for $\|R^k\| (\bB_2)$. Thus, by the usual slicing theorem,
passing to a subsequence there is a radius $\rho\in ]\frac{3}{2}, 2[$ such $\mass (\partial ((R^k-S)\res \bB_\rho))$ is uniformly bounded.
On the other hand $R^k -S$ is converging to $0$ in the sense of currents and hence, by \cite[Theorem 31.2]{Sim}, $\mathcal{F} ((R^k-S)\res \bB_\rho)\to 0$.
This means that there are integral currents $H^k, G^k$ with $\mass (H^k)+\mass (G^k)\to 0$ such that
\[
(R^k-S)\res \bB_\rho = \partial H^k + G^k\, .
\]
Taking the boundary of the latter identity we conclude that $\partial G^k = \partial ((R^k-S)\res \bB_\rho)$. Now, rescaling the almost minimality property of $T^k$,
we conclude that
\[
\|R^k\| (\bB_\rho) \leq \|S\| (\bB_\rho) + \mass (G_k) + C_1 r_k^{\alpha_0}\, .
\]
On the other hand, since $(\mass (G^k) + r_k)\downarrow 0$, we infer 
\[
\limsup_{k\to\infty} \|R^k\| (\bB_\rho) \leq \|S\| (\bB_\rho)\, .
\]
Since however $R^k\to S$ in $\bB_2$, we also have
\[
\|S\| (\bB_\rho) \leq \liminf_{k\to\infty} \|R^k\| (\bB_\rho)\, .
\]
We thus conclude that $\|R^k\| \weaks \|S\|$ on $\bB_\rho$ in the sense of measures and, since $\|S\| (\partial \bB_1) =0$ by the conical property of $S$,
we infer that $\|R^k\|(\bB_1) \to \|S\|(\bB_1)$. Thus
\eqref{e:hyp1} and \eqref{e:wlog} imply
\begin{equation}\label{hyp2}
\lim_{k\to\infty} \mass(\partial (R^k\res \bB_1)) = \mass(\partial (S\res \bB_1)) \, .
\end{equation}
The almost monotonicity formula for $T^k$ (in the rescaled version for $R^k$) implies through standard arguments that $\supp (R^k)$ converges
to $\supp (S)$ in the Hausdorff sense: one can follow, for instance, the proof of \cite[Lemma 17.11]{Sim}. Finally, again by  \cite[Theorem 31.2]{Sim}, we conclude that $\mathcal{F} ((R^k-S)\res \bB_1) \to 0$ and hence, arguing as above, we infer the existence of integer rectifiable currents $G^k$ such that
$\partial G^k = \partial ((R^k-S)\res \bB_1)$ and $\mass (G^k)\to 0$. In turn this implies
$\mathcal{F}(\partial (R^k\res \bB_1)-\de (S\res \bB_1))\to 0$. So all the assumptions of Lemma~\ref{l:construction}
are satisfied, and there exist integral currents $H^k$ such that 
$\partial H^k=\partial (R^k\res \bB_1)$
and
\begin{equation}\label{e:contra}
\|H^k\|(\bB_1)-\|S\|(\bB_1)\leq (1-\eps_{13})\big(\|0\cone \de(R^k\res \bB_1)\|(\bB_1)
-\|S\|(\bB_1)\big).
\end{equation}
By the almost minimality of $T^k$ and the usual rescaling, we conclude
\[
\|R^k\| (\bB_1) \leq \|H^k\| (\bB_1) + C_{0} r_k^{\alpha_0}\, .
\]
Thus,
\begin{align*}
\|R^k|| (\bB_1) - \|S\| (\bB_1) &\leq \|H^k\| (\bB_1) - \|S\| (\bB_1) + C_{0} r_k^{\alpha_0}\nonumber\\
&\stackrel{\eqref{e:contra}}{\leq}  (1-\eps_{13})\big(\|0\cone \de(R^k\res \bB_1)\|(\bB_1)
-\|S\|(\bB_1)\big) + C_{0} r_k^{\alpha_0}\, .
\end{align*}
However, when $k$ is so large that $\frac{1}{k} < \eps_{13}$ and $k > C_{0}$, the latter inequality contradicts \eqref{e:hyp1}
(recall \eqref{e:wlog}).
\end{proof}

\subsection{Proof of Theorem \ref{t:Uniqueness2}} Without loss of generality from now on we assume that
$x=0$ and that $\dist (0, \supp (\partial T)) \geq 2 $. Moreover we set $T_r := (\iota_{0,r})_\sharp T$.

\medskip

\textsc{Step 1. Blow-up.} By the almost monotonicity, the family $\{T_r\}_{0<r\leq 1}\subset {\bf I}_2 (\R^{n+2})$ enjoys a uniform bound for $\|T_r\| (K)$ whenever $K\subset \R^{n+2}$ is a compact set. Moreover, for any $U\subset\subset \R^{n+2}$ open, $\partial T_r \res U = 0$, provided $r$ is large enough. It follows that we can
apply the compactness theorem of integral currents, and for every sequence $r_k\downarrow 0$ we can extract a subsequence $T_{\rho_k}$ converging to an integral current $S$ with $\partial S =0$. Observe also that we can
argue as in the proof of Proposition \ref{p:generalized epi} to conclude that for every $N_0\in \N$ there is a
subsequence, not relabeled, and a $\bar r\in ]N_0, N_0+1[$ with the following properties
\begin{itemize}
\item $\|T_{\rho_k}\| (\bB_{\bar r}) \to \|S\| (\bB_{\bar r})$;
\item There are currents $H_k \in {\bf I}_2 (\R^{n+2})$ with $\mass (H^k)\downarrow 0$ and $\partial H^k =
\partial ((T_{\rho_k} -S)\res \bB_{\bar r})$.
\end{itemize}
We then easily conclude that $S$ is area minimizing in $\bB_{\bar r}$ and that $\|T_{\rho_k}\| (V) \to \|S\| (V)$ for any
open set $V\subset \subset \bB_{\bar r}$ with $\|S\| (\partial V)=0$. A standard argument shows that these properties
remain then true for every ball and for the {\em entire sequence} $\{T_{\rho_k}\}$. As a consequence
of the fact that $\Theta (0, T)$ exists, we then conclude that
\[
\|S\| (\bB_r (0)) = \Theta (T, 0) r^2 := Q \omega_2 r^2
\]
for all radii but an (at most) countable family (recall that $\omega_2$ denotes the area of the unit disk in $\R^2$). It is then a standard fact, using the monotonicity formula for area-minimizing currents, that $S$ is a cone (see for instance \cite{Sim}). Finally, it is well known that $2$-dimensional area minimizing cones are all sum of planes intersecting only at the origin (see for instance \cite{Fl}). So we conclude from the standard theory of currents (see for instance the proof of Proposition \ref{p:generalized epi}) that
$\mathcal{F} ((T_{\rho_k}- S)\res \bB_r)  \to 0$ for every $r>0$. 

Let $\eps_{11}$ be the constant of Proposition \ref{p:generalized epi}. We then conclude the existence of a radius $r_0>0$ such that, for every $r< r_0$ there is an
an area minimizing cone $S$ such that $\mathcal{F} ((T_{2r}-S)\res \bB_1) \leq 2\,\eps_{11}$.
We can then apply \eqref{e:generilez epi} for every $0<r<r_0$ such that
$\de (T\res \bB_r) \in {\bf I}_1(\de \bB_r)$ (which holds for a.e.~$r$).
After scaling back and multiplying by $r^2$, we get
\begin{equation}\label{e:diff ineq 1}
\mass(T\res \bB_r) - Q\,\omega_2\,r^2 \leq \left(1-\eps_{12}\right) \Big( \mass(0\cone \de (T\res \bB_r)) - Q\,\omega_2\,r^2\Big) + \bar c\, r^{2+\alpha_0}\quad \text{for a.e.}\; r<r_0\, .
\end{equation}
Set $f(r) := \mass(T\res \bB_r) - Q\,\omega_2\,r^2$.
Since $r \mapsto \mass(T\res \bB_r)$ is monotone, the function
$f$ is differentiable a.e. and its distributional derivative is a measure. Its absolutely continuous part
coincides a.e. with the classical differential and its singular part is nonnegative. Note also that we can assume
$2+\alpha_0 > \eps + \frac{2}{1-\eps} =: \eps + a$ for some $\varepsilon >0$. 

Therefore, by the well-known expansion for the mass of a cone, 
\eqref{e:diff ineq 1} reads
\begin{equation}\label{e:diff ineq 2}
-a\,\bar c\, r^{\eps-1} \leq \frac{d}{dr} \big(r^{-a} f(r) \big)\, ,
\end{equation}
Integrating  \eqref{e:diff ineq 2} we get $-\frac{a}{\eps}\, \bar c\, \big(r^{\eps} - s^{\eps}\big) \leq r^{-a} f(r) - s^{-a} f(s)$ for all $0 < s < r < r_0$.
Setting $e(r) := \frac{f(r)}{\omega_2 r^2}$ this implies
\begin{equation}\label{e:decay}
e(s) \leq \left(\frac{s}{r}\right)^a \, e(r)
+ C\,r^\eps
\qquad \forall\; 0 < s < r < r_0.
\end{equation}

\medskip

\textsc{Step 2.} Consider now the map $F(x):= \frac{x}{|x|}$ and
radii $0 < \frac{t}{2} \leq s \leq t < r_0$.
By the area formula, 
\begin{align}
\mass(F_\sharp (T\res (\bB_t\setminus \bB_s))) & \leq \int_{\bB_t\setminus \bB_s}\frac{|x^\perp|}{|x|^3}\,d\|T\|\notag\\
& \leq \underbrace{\left(\int_{\bB_t\setminus \bB_s}\frac{|x^\perp|^2}{|x|^4}\,d\|T\|\right)^{\sfrac{1}{2}}}_{:=I_1}
\cdot\underbrace{\left(\int_{\bB_t\setminus \bB_s}\frac{1}{|x|^2}\,d\|T\|\right)^{\sfrac{1}{2}}}_{I_2}.\notag
\end{align}
$I_1$ and $I_2$ can be easily estimated using the almost monotonicity
formula
\begin{gather}
I_1^2 \stackrel{\eqref{e:monotonicity formula}}{\leq} 
e(t) - e(s) +C_1\, t^{\alpha_0} \leq e(t) + 2\,C_1\, t^{\alpha_0} 
\stackrel{\eqref{e:decay}}{\leq} C\, t^{\sfrac{\eps}{2}}, \label{e:I1}\\
I_2^2 \leq \frac{\|T\|(\bB_t)}{s^2} \stackrel{\eqref{e:monotonicity formula}}{\leq}
\left(\frac{t}{s}\right)^2
\left[\frac{\|T\|(\bB_{r_0})}{r_0^2} + C_1\,r_0^{\alpha_0}\right]
\leq C,\label{e:I2}
\end{gather}
where we took into account that, by \eqref{e:monotonicity formula},
$e(s) > -C_1 s^\alpha$ for every $s>0$ and that $C>0$ is a constant depending on $r_0$.
In particular we conclude that
\[
\mass(F_\sharp (T\res (\bB_t\setminus \bB_s)))  \leq C \, t^{\sfrac{\eps}{2}}
\qquad \forall\; 0 < \frac{t}{2} \leq s \leq t < r_0,
\]
and, by iteration on diadic intervals,
\begin{equation}\label{e:primo decay}
\mass(F_\sharp (T\res (\bB_r\setminus \bB_s)))  \leq C\, r^{\sfrac{\eps}{2}}
\qquad \forall\; 0 < s < r < r_0.
\end{equation}
Since $\de F_\sharp (T\res (\bB_r\setminus \bB_s)) =
\de (T_r\res \bB_1) - \de (T_s\res \bB_1)$ for a.e.~$0<s<r$, from the definition of $\mathcal{F}$ we get:
\begin{equation}\label{e:ultimo rate}
\mathcal{F}\big(\de (T_r\res \bB_1) - \de (T_s\res \bB_1)\big)  \stackrel{\eqref{e:primo decay}}{\leq}
C\, r^{\sfrac{\eps}{2}}.
\end{equation}
This implies that the currents $\partial (T_r \res \bB_1)$ converges to a unique current $Z$. On the other hand,
by the almost monotonicity formula it follows easily that $T_r \res \bB_1$ converge to the cone $0\cone Z$. Since we already
know that an appropriate sequence converges to $S = \sum_i n_i \llbracket \pi_i \rrbracket$, we conclude that $T_r$ converges to
$S$.

\medskip

\textsc{Step 3. Proof of \eqref{e:roc1} and \eqref{e:roc2}.}
In order to prove \eqref{e:roc1},
it is enough to find integral currents $V$ and $W$ such that
$T_r - T_s = \de H + W$ and $\mass(H) + \mass(W) \leq C r^{\sfrac{\eps}{2}}$.
To this aim, fix a small parameter $a>0$. Let
$\llbracket p,q\rrbracket$ denote
the current in $\mathbf{I}_1 (\R)$ induced by the oriented segment $\{t: p\leq t \leq q\}$.  Similarly $\a{p}\in \mathbf{I}_0 (\R)$ 
is the Dirac mass at the point $p$.
Consider the currents $V_{a} \in \mathbf{I}_3(\R\times \R^{n+2})$ defined by
\[
V_{a}:=\Big(\a{0,1}\times T\res (\bB_r\setminus \bB_a)\Big)\res
\Big\{(t,x)\in \R\times \R^{n+2}\,:\, r^{-1}|x| \leq t\leq s^{-1}|x|\Big\}.
\]
Next, we consider the map $h\colon \R\times (\R^{n+2}\setminus\{0\})\ni (t,x)\to \frac{t\,x}{|x|}
\in \R^{n+2}$
and the currents $H_{a}:=h_\sharp V_{a}$. If $d_1,d_2:\R^{}\times \R^{n+2} \to \R$ denote the functions
$d_1(t,x) := t-s^{-1}|x|$ and $d_2(t,x) := t-r^{-1}|x|$, then
for a.e.~$a>0$ we have
\begin{align*}
\partial V_{a}={}& \a{1} \times T\res (\bB_{r}\setminus \bB_s)-
\a{\textstyle{\frac{a}{r}},\textstyle{\frac{a}{s}}}\times \partial (T\res \bB_a)\\
&+\langle\a{0,1}\times T\res(\bB_r\setminus \bB_a), d_1, 0\rangle
- \langle\a{0,1}\times T\res(\bB_r\setminus \bB_a), d_2, 0\rangle\, .
\end{align*}
Since $\partial$ commutes with the push-forward, we also get
\begin{equation}\label{e:homotopy}
\partial H_a= F_\sharp (T\res (\bB_r\setminus \bB_s) -
\underbrace{h_\sharp\left(\a{\textstyle{\frac{a}{r}},\textstyle{\frac{a}{s}}}\times \partial (T\res \bB_a)\right)}_{Z_a} -
T_r\res (\bB_1\setminus \bB_{\frac{a}{r}})+
T_s\res (\bB_1\setminus \bB_{\frac{a}{s}}),
\end{equation}
where we have used the fact that
$h(t,x)\equiv s^{-1} x$ and 
$h(t,x)\equiv r^{-1} x$ respectively in the sets
$\{(t,x)\in \R\times (\R^{n+2}\setminus \{0\}) \,:\, t=s^{-1}|x|\}$
and $\{(t,x)\in \R\times (\R^{n+2}\setminus \{0\}) \,:\, t=r^{-1}|x|\}$.
It is simple to see that there exists $H$ such that
$H_{a} \to -H$ as $a\downarrow 0$. Thus \eqref{e:homotopy} gives
\[
- \partial H = F_\sharp (T\res (\bB_r\setminus \bB_s)) -
T_r\res \bB_1+ T_s\res \bB_1, 
\]
because $\mass(Z_a) \leq a\, |s^{-1}-r^{-1}| \mass (\de (T_a\res \bB_1))
\leq C\, a\, |s^{-1}-r^{-1}| \mass (\de (T_0\res \bB_1)) \to 0$.
To conclude \eqref{e:roc1} we only need to estimate the mass of $H$.
To this extent, note that
$h_\sharp (\frac{\de}{\de t} \wedge \vec T)= dh (\frac{\de}{\de t}) \wedge  h_\sharp (\vec T)$
and, since $dh\left(\frac{\de}{\de t}\right)=\frac{x}{|x|}$,
\begin{eqnarray*}
H ( \omega) &=& \int_0^1 \int_{\bB_{rt}\setminus \bB_{st}} \langle h_\sharp \left( \textstyle{\frac{\partial}{\partial t}} \wedge \vec{T}\right) , \omega_{h(x)}\rangle d\|R\| (x)\, dt\nonumber\\
&=& \int_0^1 \int_{\bB_{rt}\setminus \bB_{st}} \langle \textstyle{\frac{t x}{|x|}} \wedge (F_{\sharp} \vec{T}), \omega_{tx/|x|}\rangle\,  d\|T\| (x)\, dt
\nonumber\\
&=& \int_0^1 \int_{\bB_{rt}\setminus \bB_{st}} \langle (t F)_\sharp \vec{T}, \omega_{tx/|x|} \ser {\textstyle{\frac{x}{|x|}}}\rangle d\|T\| (x)\, dt\\
&=& \int_0^1 (tF)_\sharp \left(T \res (\bB_{rt}\setminus \bB_{st})\right)  (\omega\ser \textstyle{\frac{x}{|x|}})\, dt
\end{eqnarray*}
Thus
\begin{eqnarray*}
\mass (H) &\leq& \int_0^1 \mass ((tF)_\sharp  \left(T \res (\bB_{rt}\setminus \bB_{st})\right) dt
= \int_0^1 t^2 \mass (F_\sharp (T \res (\bB_{rt}\setminus \bB_{st})))\, dt\\
&\stackrel{\eqref{e:primo decay}}{\leq}& C \int_0^1 r^{\eps/2} t^{2+\eps/2}\, dt \leq C r^{\eps/2}\, .
\end{eqnarray*}
\eqref{e:roc2} follows then from the almost monotonicity formula, see for instance \cite[Lemma 17.11]{Sim}.

\medskip

\textsc{Step 4. Decomposition.} We first introduce the following notation: we call $T$ irreducible in $\bB_r (x)$ if it is not possible to find two 
(integral) currents with $T\res \bB_r(x)=T^1+T^2$ and $\supp (T^1)\cap \supp (T^2)=\{0\}$ (cf. to the notion of {\em indecomposabality} as in \cite[4.2.25]{Fed}: $T$ is indecomposable if it is impossible to write it as $T^1 + T^2$ with $\partial T_1 \res \bB_r (x) = \partial T_2 \res \bB_r (x) = 0$ and $\mass (T_1)+\mass (T_2)= \|T\| (\bB_r (x))$). If $T$ is reducible, then clearly
$\Theta(\|T\|, x)=\Theta(\|T^1\|,x)+\Theta(\|T^2\|,x)$. Since each $T^i$ would be almost minimizing,
$\Theta(\|T^i\|,x)\in \N\setminus \{0\}$ and we can only decompose $T$ finitely many times.
Next suppose by contradiction that $T$ is irreducible in $x$ but its tangent cone $T_{x,0}$ is not a plane. Then, since $T_{x,0}$ is area minimizing,
by \cite{Fl}, there exists $J\geq 2$ such that  
$T_{x,0}=\sum_{i=1}^JQ_i\a{V_i}$, 
where $V_i\subset \R^{n+2}$ are $2$-dimensional linear subspaces such that $V_i\cap V_j=\{0\}$ for every $i\neq j$ and $Q_i\in \N$ satisfy $\sum_{i=1}^J Q_i=Q$. Then consider the currents 
\[
T^i:=T\res \{y\in \R^{m+n}\,:\, \dist (y-x, V_i) \leq C r^{1+\gamma}\}
\quad\text{for } i=1,2, \ldots, J\, . 
\]
By \eqref{e:roc2} this is a decomposition of $T$ in two non-zero currents whose supports intersect each other only in $\{0\}$, which is a contradiction.

\section{Proof of Lemma \ref{l:construction}}

As already mentioned, any $2$-dimensional area-minimizing cone $S$ is the sum of (integer multiples) of finitely many oriented planes, each pair of which
intersects only at the origin. Therefore the support of the cycle $R := \partial (S\res \bB_1)$ of the statement of Lemma \ref{l:construction} consists of a finite number (say $N$) of disjoint equatorial circles of $\partial \bB_1$. By condition (iii), we can thus assume that $Z$ splits into $N$ cycles, each close (in the sense of (i), (ii) and (iii)) to an integer multiple of an equatorial circle of $\partial B_1$. Thus, without loss
of generality, from now on we assume that $S$ is given by $Q \a{\pi_0}$, where $\pi_0$ is the (oriented) plane $\R^2\times \{0\}\subset \R^{n+2}$ and $Q$ is a positive integer. Correspondingly, $R = Q \a{\gamma_0}$ where $\gamma_0$ is the oriented equatorial circle
$\pi_0 \cap \partial \bB_1$. 

\medskip

{\sc Step 1. Reduction to a Lipschitz winding curve.} We next introduce the notation $B_r (x, \pi)$ for the $2$-dimensional disk $x+ \bB_r (0) \cap \pi$ and $\bC_r (x, \pi)$ for the cylinder $B_r (x, \pi) + \pi^\perp$, omitting $x$ when it is the origin and $\pi$ when it is the plane $\pi_0$.  Given any $1$-dimensional cycle $W$ we consider the infinite $2$-dimensional cone $T$ with vertex $0$ and spherical cross section $W$, namely $\lim_{R\to \infty} (\iota_{0,R})_\sharp (0\cone W)$ and denote it by $(0\cone W)_\infty$. The {\em cylindrical excess} of any infinite $2$-dimensional cone $T$ in $\bC_1 (\tau)$ is then given by
\[
\bE (T, \tau) := \frac{1}{2} \int_{\bC_1 (\tau)} |\vec{T} (x) -\tau|^2 \, d\|T\| (x)
\]
whereas the cylindrical excess of $Z$ is
\[
\bE (Z) := \min_\tau \bE ((0\cone Z)_\infty, \tau)\, .
\]
It is simple to see that under the assumptions (i), (ii) and (iii), any minimum point $\tau$ for $(0\cone Z)_\infty$ in the expression above must be close to $\pi_0$. 

Let now $\bP$ be the orthogonal projection onto $\partial \bB_1$ (which obviously it is defined in $\R^{n+2}\setminus \{0\}$). For each $\pi$, such projection is invertible when we restrict its domain of definition to $\partial \bC_1 (\pi)$ and its target to $\partial \bB_1 \setminus \pi^\perp$. We then let $\bP_\pi^{-1}$ be its inverse. Note also that, under the assumptions (i), (ii) and (iii), when $\tau$ is close enough to $\pi_0$, $\supp (Z)\subset \bB_1 \setminus \tau^\perp$. Therefore, for
any such $\tau$ we have 
\[
(0\cone Z)_\infty \res \bC_1 (\tau)= 0 \cone (\bP_\tau^{-1})_\sharp Z\, .
\] 
In particular such identity is valid for the $\pi$ which minimizes
$\bE ((0\cone Z)_\infty, \tau)$. If $Z$ is as in the statement of the lemma, by a well-known result in geometric measure theory, $Z$ can be written as the sum of (at most countably many) $1$-dimensional cycles $Z_i$, where each $Z_i$ is a simple closed Lipschitz curve and $\sum \mass (Z_i) = \mass (Z)$. Observe also that, if $\eps$ is sufficiently small, then $(\proj_{\pi_0})_\sharp (\bP_{\pi_0}^{-1})_\sharp Z_i$ (where
$\proj_{\pi_0}$ is the orthogonal projection onto $\pi_0$) equals $k_i \a{\gamma_0}$ for some nonnegative integer $k_i$. We thus have $\sum k_i = Q$ and it follows by standard arguments that each $Z_i$ fulfills the assumptions (i), (ii) and (iii) of the Lemma with $k_i$ in place of $Q$ and with $\eps' >0$ in place of $\eps$, where the constant $\eps'\downarrow 0$ as $\eps\downarrow 0$. Thus, it suffices to prove the main estimate for each $Z_i$ and sum it over $i$.  Observe next that assumption (ii) in the Lemma excludes the possibility that $k_i < 0$ for some $i$. Moreover, the case $k_i=0$ corresponds to the trivial situation in which the minimizing cone $S$ is $0$. In this case 
$\mass (Z_i) < \eps_{13}$ and we can use the the isoperimetric inequality to find an $H$ such that $\partial H = Z_i$ and
\[
\|H\| (\bB_1) \leq C (\mass (Z))^2 \leq C \eps_{13} \mass (Z) \leq C \eps_{13} {\textstyle{\frac{1}{2}}} \|0\cone Z\| (\bB_1)\, .
\]
It suffices therefore to consider the case $k_i >0$.

Summarizing, in addition to (i), (ii) and (iii) we can also assume, w.l.o.g., the following:
\begin{itemize}
\item[(iv)] $R = Q \a{\gamma_0}$ for some integer $Q>0$;
\item[(v)] $Z = \eta_\sharp \a{[0, \mass (Z)]}$, where $\eta: [0, \mass (Z)] \to \partial \bB_1$ is Lipschitz and $\eta (0) = \eta (\mass (Z))$;
\item[(vi)] If $\bE ((0\cone Z)_\infty, \tau) = \bE (Z)$, then $\bE (Z, \tau) < \bar\eps$ and $(\proj_{\tau})_{\sharp} (\bP_{\tau}^{-1})_\sharp Z = Q \a{\gamma_0}$
(where $\bar{\eps} (\eps, Q) \downarrow 0$ as $\eps\downarrow 0$).
\end{itemize}

For any fixed $\delta>0$, we next use \cite[Proposition 2.7]{Wh} to find a second curve $\zeta' : [0, 2Q \omega_2] \to \partial \bC_1 (\tau)$ with the following properties (recall that $2\omega_2$ is the length of the unit circle in $\R^2$):
\begin{itemize}
\item[(a1)] $\zeta' (\vartheta) = (\cos \vartheta, \sin \vartheta, f' (\vartheta))\in \tau \times \tau^\perp$ for some Lipschitz function $f': [0, 2Q\omega_2]\to \tau^\perp$ with $f'(0)= f' (2Q\omega_2)$ and $\|f'\|_\infty + \Lip (f') \leq \delta$;
\item[(a2)] If we set $Z' = \zeta'_\sharp \a{[0, 2Q\omega_2]}$, then $\mass ((\bP_{\tau}^{-1})_\sharp Z - Z') \leq \bE (Z)/ C (\delta)$;
\item[(a3)] $\bE ((0\cone Z')_\infty, \tau) \leq \bE ((0\cone Z)_\infty, \tau) = \bE (Z)$.
\end{itemize} 
$\delta$ will be chosen (sufficiently small) later.  Since from (a2) we conclude easily
\[
\mass (((0\cone Z)_\infty - (0\cone Z') _\infty)\res \bC_2 (\tau')) \leq \bE (Z) / C (\delta)\, , 
\]
we also infer
\[
\mass (\partial ((0\cone Z - 0 \cone Z')\res \bC_{1/2} (\tau'))) \leq \bE (Z) / C(\delta)\, .
\]
After applying a rotation we can assume that $\tau' = \pi_0$. We thus achieve, in addition to (i)-(vi), the condition
\begin{itemize}
\item[(vii)] $\bE ((0\cone Z')_\infty, \pi_0) = \bE ((0\cone Z')_\infty)$ and $\mass (\partial ((0\cone Z - 0 \cone Z')\res \bC_{1/2})) \leq \bE (Z) / C(\delta)$.
\end{itemize}
Next, observe that if $\tau'$ minimizes $\bE ((0\cone Z')_\infty, \tau')$, then 
\[
|\tau' - \tau| \leq  C \, \bE((0\cone Z')_\infty, \tau)\leq C \bE(Z) \leq C\,\eps
\] 
for some geometric constant $C$. Hence elementary considerations (see for instance the reparametrization Lemma \cite[Lemma B.1]{DS4}) lead easily to the following conclusions:
\begin{itemize}
\item[(viii)] the cycle $Z'':= \partial  ((0\cone Z') \res \bC_{1/2})$ is of the form
$\zeta_\sharp \a{[0, 2Q\omega_2]}$ for some $\zeta (\vartheta) = \frac{1}{2} (\cos \vartheta, \sin \vartheta, f (\vartheta))\in \pi_0\times \pi_0^\perp$, where
$|f| + \Lip (f) \leq C \delta$ ($C$ being a geometric constant);
\item[(ix)] $\bE ((0\cone Z'')_\infty, \pi_0) = \bE (Z'') \stackrel{(a3)\&(vii)}{ \leq} \bE (Z) < \bar\eps$.
\end{itemize}

\medskip

{\bf Step 2. Cylindrical epiperimetric inequality and conclusion.} Consider the Fourier expansion of $f$ as 
\[
f (\vartheta) = \alpha_0 + \sum_{i=0}^\infty \left(\alpha_i \cos \left(\textstyle{\frac{i}{Q}} \vartheta\right) + \beta_i
\sin \left(\textstyle{\frac{i}{Q}} \vartheta\right)\right)\, 
\]
and let
\[
P (f) := \alpha_Q \cos + \beta_Q \sin\, .
\]
We first claim the existence of a constant $K$ (depending only upon $Q$) such that, provided $\delta$ is smaller than some geometric constant, then 
\begin{equation}\label{e:Fourier}
\|(f - P (f))\|_{W^{1,2}} \geq K \|f\|_{W^{1,2}}\, .
\end{equation}
Indeed consider the $2$-dimensional plane $\tau$ which contains the image of the map $\vartheta \mapsto (\cos \vartheta, \sin \vartheta, P (f) (\vartheta))$. It is then straightforward to check that 
\begin{itemize}
\item $\bC_1 (\tau) \cap \supp ((0\cone Z'')_\infty) \subset \bC_2$
\item If $x= r \zeta (\vartheta) \in \supp (Z'')$ and $r>0$, then
\begin{align}
&|\vec T (x) - \pi_0| \geq \frac{1}{C} \left(|Df (\vartheta)| + |f (\vartheta)|\right)\\
&|\vec T (x) - \tau| \leq C \left(|D (f - P (f)) (\vartheta)| + |(f - P (f)) (\vartheta)|\right)\, ,
\end{align}
where $C$ is just a geometric constant.
\end{itemize}

Using that $\Lip (f) \leq \delta$, by the area formula we easily conclude that
\begin{align}
\bE ((0\cone Z'')_\infty, \pi_0) &\geq \frac{1}{C} \|f\|_{W^{1,2}}^2\\
\bE ((0\cone Z'')_\infty, \tau) &\leq C \|f - P (f)\|_{W^{1,2}}^2\, .
\end{align}
Since $C$ is a fixed geometric constant, \eqref{e:Fourier} follows easily from 
\[
\bE ((0\cone Z''), \pi_0) = \bE (Z'') 
\leq \bE ((0\cone Z'')_\infty, \tau)\, .
\]

Next, following \cite[Proposition 2.4]{Wh} we consider the map $g: ]0, \frac{1}{2}]\times [0, 2Q\omega_2] \to \R^n$ given by
\[
g (r, \vartheta) =  \alpha_0 + \sum_{i=0}^\infty r^{\sfrac{i}{Q}} \left(\alpha_i \cos \left(\textstyle{\frac{i}{Q}} \vartheta\right) + \beta_i
\sin \left(\textstyle{\frac{i}{Q}} \vartheta\right)\right)
\]
and let $H' = g_{\sharp} \a{]0, \frac{1}{2}]\times [0, 2Q\omega_2]}$. By \cite[Proposition 2.4]{Wh} we have $\partial H' = Z''$ and 
\[
\mass (H') - \frac{Q}{4}\omega_2 \leq \frac{1}{4} (1- 8 \eps_{13}) \bE (Z'') \leq \frac{1}{4} (1-8\eps_{13}) \bE (Z)\, ,
\]
for some $\eps_{13} (Q, K)>0$.

Next, using the isoperimetric inequality we find a $2$-dimensional current $K$ such that $\partial K = \partial ((0\cone Z)\res \bC_{1/2}) - Z''
= \partial ((0\cone Z - 0 \cone Z')\res \bC_{1/2})$
and 
\[
\mass (K) \leq C (\mass (\partial ((0\cone Z)\res \bC_{1/2}) - Z''))^2 \stackrel{(vii)}{\leq} C (\delta) \bE (Z)^2\, .
\]
Thus, if we set $H := H' + K + 0 \cone Z \res \bB_1 \setminus \bC_{1/2}$, we have $\partial H = Z$ and
\[
\mass (H) \leq \frac{Q}{4} \omega_2 + \frac{1}{4} (1-8\eps_{13}) \bE (Z) + C (\delta) \bE (Z)^2 + \mass ((0\cone Z)\res \bB_1 \setminus \bC_{1/2})\, .
\]
Since $\bE (Z) < \bar\eps$, it suffices to choose $\eps$ sufficiently small to achieve
\[
\mass (H) \leq \frac{Q}{4} \omega_2 + \frac{1}{4} (1-4\eps_{13}) \bE (Z) + \mass ((0\cone Z)\res \bB_1 \setminus \bC_{1/2})\, .
\]
Next recall that 
\begin{align*}
\frac{1}{4} \bE(Z) \leq & \frac{1}{4} (\bE ((0\cone Z)_\infty, \pi_0)  = \frac{1}{8} \int_{\bC_1} |\vec{T} - \pi_0|^2 d\|0\cone Z\|\nonumber\\
= & \frac{1}{4} (\mass ((0\cone Z)\res \bC_1) - Q \omega_2) = \mass ((0\cone Z)\res \bC_{1/2}) - \frac{Q\omega_2}{4}\, ,
\end{align*}
where the first equality in the last line is due to ${\proj_{\pi_0}}_\sharp (0\cone Z) = Q \a{B_1 (0, \pi_0)}$. We therefore infer
\begin{align*}
\mass (H) - Q \omega_2 &\leq \mass (0\cone Z) + \eps_{13} Q\omega_2- 4\eps_{13} \mass ((0\cone Z)\res \bC_{1/2}) - Q \omega_2\\
&\leq \mass (0\cone Z) + \eps_{13} Q\omega_2 - 4 \eps_{13} \mass ((0\cone Z) \res \bB_{1/2}) - Q\omega_2\\
&= \mass (0\cone Z) +\eps_{13} Q \omega_2 - \eps_{13} \mass (0\cone Z) - Q \omega_2\\
& = (1-\eps_{13}) (\mass (0\cone Z) - Q \omega_2)\, .
\end{align*}

\bibliographystyle{plain}
\bibliography{references-TU}

\end{document}